\documentclass{birkjour}
\usepackage{cases}
\usepackage{amsmath}
\usepackage{amssymb}

 \usepackage{cases}
 \newtheorem{thm}{Theorem}[section]
 
 \newtheorem{lem}[thm]{Lemma}
 
 \theoremstyle{definition}
 \newtheorem{defn}[thm]{Definition}
 \theoremstyle{remark}
 
 \newtheorem{ex}{Example}
 \numberwithin{equation}{section}

\begin{document}

%
%
%
%
%
%
%
%
%

\title[Existence and uniqueness of fixed point for ordered contraction type ]
 {Existence and uniqueness of fixed point for ordered contraction type operator in Banach Space
  }


\author[Jinxiu Mao]{Jinxiu Mao}

\address{
Department of Mathematics,
 Qufu Normal University, Qufu,
 Shandong, 273165, P. R. of
China}

\email{maojinxiu1982@163.com}

\thanks{Supported by the Natural Science Foundation of China(11571197), Natural Science Foundation of Shandong Province(ZR2014AM007) and the Science
Foundation of Qufu Normal University of China (XJ201112).}
\author{Zengqin Zhao}
\address{School of Mathematics,
 Qufu Normal University, Qufu,
 Shandong, 273165, P. R. of
China}
\email{zqzhaoy@163.com}
\subjclass{47H10}

\keywords{Fixed point, Decreasing operator, Cone.}

\date{September 15, 2014}
\dedicatory{}

\begin{abstract}
In this paper, we investigate the existence and
uniqueness of fixed point for partially ordered contraction type
operators in Banach Space. We also present applications to integral and differential equations.
\end{abstract}

\maketitle
\section{Introduction}
Existence of fixed points for contraction type maps in partially ordered metric space has been considered recently in \cite{c}-\cite{M}, where some applications to matrix equation,
ordinary differential equations and integral equations are presented, see \cite{ZHM}-\cite{HQ3}. The following generalization of Banach's contraction principle is due to Geraghty \cite{MG}.

Let $\zeta$ denotes the class of those functions $\beta: [0, \infty)\rightarrow [0,1)$ which satisfy the condition
\begin{equation}\beta(t_{n})\rightarrow 1 \Rightarrow t_{n}\rightarrow 0. \end{equation}
\begin{thm}
 Let $(M,\ d)$ be a complete metric space and let $f: M\rightarrow M$ be a map. Suppose there exists $\beta\in \zeta$ such that for
each $x,\ y\in M$,
\begin{equation}d(f(x),f(y))\leq \beta(d(x,y))d(x,y). \end{equation}
Then f has a unique fixed point $z\in M$, and $\{f^{n}(x)\}$ converges to $z$, for each $x\in M$.
\end{thm}
In \cite{JK}, J. Harjani and K. Sadarangani studied fixed point theorems for weakly contractive mappings in partially
ordered sets.
Very recently, Amini-Harandi and Emami \cite{AH} proved the following existence theorem which is a version of Theorem 1.1 in the context of partially ordered complete metric spaces and a generalization of results in \cite{JK}:
\begin{thm}
 Let $(M, \preceq)$ be a partially ordered set and suppose that there exists a metric $d$ in $M$ such that $(M, \ d)$ is a complete
metric space. Let $f: M\rightarrow M$ be an increasing map such that there exists an element $x_0\in M$ with $x_0\preceq f(x_0)$. Suppose
that there exists $\beta\in \zeta$ such that
\begin{equation}d(f(x),f(y))\leq \beta(d(x,y))d(x,y), \ \ x,\ y \in M, y\preceq x. \end{equation}
Assume that either $f$ is continuous or $M$ is such that
\begin{equation*}
\text{if an increasing sequence}\  \{x_{n}\}\rightarrow x\  \text{in}\  M, \text{then}\  x_{n}\preceq x, \forall n.
\end{equation*}
Besides, if
\begin{equation*}
\text{for each}\  x,y\in M, \text{there exists}\  z\in M, \text{which is comparable to}\  x \text{and}\  y.
\end{equation*}
Then f has a unique fixed point.
\end{thm}
 In this paper, we generalize Theorem 1.2 from three aspects. Firstly, the contraction condition (\ref{eq1}) is merely about partial order, while in (1.3) the contraction is about metric and $\beta: [0, \infty)\rightarrow [0,1)$. The major difficult brought by (\ref{eq1}) is that in (\ref{eq*}) the contraction constant $Nf(\parallel u-v\parallel)$ may bigger than 1, as the normal constant $N$ of a cone is bigger than 1, see \cite{Sh} in Lemma 2.1. Secondly, we do not need continuity or the equivalent condition of the operator as in \cite{NR}-\cite{AH}. Thirdly, we don't need any upper or lower solution as in \cite{NR}-\cite{AH}.
 Our methods are different from that in \cite{AH}. In Section 3, an application to an integral equation is given.

Let us recall some preliminaries first.
\begin{defn}[\cite{GL}]\label{l1}
Let $E$ be a real Banach space. A nonempty convex closed set $P\subset E$ is called a cone if

$(i)$ $x\in P,\ \lambda \geq0\Rightarrow \lambda x\in P;$

$(ii)$ $x\in P,\ -x\in P\Rightarrow x=\theta,$ $\theta$ is the zero element in $E$.
\end{defn}
In the case that $P$ is a given cone in a real Banach space $(E, \parallel.\parallel)$, a partial
order "$\leq$" can be induced on $E$ by $x\leq y\Leftrightarrow y-x\in P.$ The cone $P$ is called
normal if there exists a constant $N > 0$, such that for all $x, y\in E,
\theta\leq x\leq y$ implies
that $\parallel x \parallel\leq N\parallel y\parallel$.
 The minimal such number $N$ is called the normal constant of $P$.
Details about cones and fixed point of operators can be found in \cite{GL}-\cite{B}.

\begin{lem}[\cite{GL}]\label{l1} A cone $P$ is normal if and only if
there exists a norm $\parallel.\parallel_{1}$ in $E$ which is
equivalent to $\parallel.\parallel$ such that for any $\theta\leq
x\leq y, \parallel x\parallel_{1}\leq \parallel
y\parallel_{1}$, i,e., $\parallel.\parallel_{1}$ is monotone. The equivalence of $\parallel.\parallel$ and $\parallel.\parallel_{1}$  means that there
exist $M>m>0$ such that $m\parallel.\parallel_{1}\leq
\parallel.\parallel\leq M\parallel.\parallel_{1}$.
\end{lem}

\begin{lem}[\cite{GL}]\label{l1} Let $P$ be a normal cone in a real Banach space $E$. Suppose that $\{x_{n}\}$ is a monotone sequence
 which has a subsequence $\{x_{n_{i}}\}$ converging to $x^{\ast}$, then $\{x_{n}\}$ also converges to $x^{\ast}$. Moreover, if $\{x_{n}\}$
 is an increasing sequence, then $\{x_{n}\}\leq x^{\ast} (n=1,2,3,...)$; if $\{x_{n}\}$
 is a decreasing sequence, then $x^{\ast}\leq \{x_{n}\} (n=1,2,3,...)$.
\end{lem}
\section{Main results}
We suppose that $E$ is an partially ordered Banach
space. $P$ is a normal cone and the normal
constant is $N$. The partial order "$\leq$" on $E$ is induced by the cone $P$.
\begin{thm}[Main Theorem]
 Suppose that $A:E\rightarrow E$ is a decreasing operator and
satisfies the following ordered contraction type condition:

(H) There exists an increasing function $f:(0,
+\infty)\rightarrow(0,1)$ such that
\begin{equation}\label{eq1}Au-Av\leq f(\parallel v-u\parallel)(v-u),\ \ \forall u,v\in E,\ u\leq v. \end{equation}
Besides, if
\begin{equation}\label{eq}
\text{for each}\  x,y\in E, \text{there exist both}\  \inf\{x,y\}\  \text{and}\  \sup\{x,y\}.
\end{equation}
Then $A$ has unique fixed point in $E$.
\end{thm}
\begin{proof}
Let $u_{0}\in E,$ we have $Au_{0}\in E.$ So we have the following two cases.

\textbf{Case I}: When $u_{0}$ is comparable to $Au_{0.}$ Firstly,
without loss of generality, we suppose that
\begin{equation}
 u_{0}\leq Au_{0}.
\end{equation}
If $u_{0}=Au_{0},$ then the proof is finished. Suppose that $u_{0}<
Au_{0}.$ Since $A$ is decreasing we obtain $Au_{0}\geq
A^{2}u_{0}$ and it is easy to prove that $A^{2}$ is increasing. Using the
contractive condition (\ref{eq1}), we have
\begin{equation}Au_{0}-A^{2}u_{0}\leq f(\parallel Au_{0}-u_{0}\parallel)(Au_{0}-u_{0})\leq Au_{0}-u_{0}.\end{equation}
So $A^{2}u_{0}\geq u_{0},$ that is \begin{equation}\label{eq3}u_{0}\leq A^{2}u_{0} .\end{equation} From
(\ref{eq1}) and the normality of cone $P$, we have
\begin{equation}\label{eq*}\parallel Au-Av\parallel\leq Nf(\parallel v-u\parallel)\parallel v-u\parallel,\ \ \forall u,v\in E,\ u\leq v,\end{equation}
\begin{align*}
A^{2}v-A^{2}u & \leq f(\parallel Au-Av\parallel)(Au-Av)\\
&\leq f(\parallel Au-Av\parallel) f(\parallel u-v\parallel)(v-u)\\
&\leq f(Nf(\parallel u-v\parallel)\parallel u-v\parallel)f(\parallel u-v\parallel)(v-u).
\end{align*}
Let $A^{2}=B.$ From (\ref{eq3}) and the above inequalities we have the following two conclusions:

(a) There exists a nondecreasing function $f:(0,
+\infty)\rightarrow(0,1)$ such that for $u,v\in E$ with $u\leq v$
\begin{equation}\label{eq7}Bv-Bu\leq f(Nf(\parallel u-v\parallel)\parallel u-v\parallel)f(\parallel u-v\parallel)(v-u),\end{equation}

(b) There exists $u_{0}\in E$ such that $u_{0}\leq Bu_{0}.$

We assert that the operator $B$ has unique fixed point in $E$. In fact, we can
use the method of iteration to construct the fixed point of $B$. Consider the iterative sequence
\begin{equation}x_{n+1}=Bx_{n},\ \ n=0,1,2,\cdots.\end{equation}
Since $x_{0}\leq Bx_{0}$ and the operator $B$ is increasing, we
have
\begin{equation}\label{eq5}x_{0}\leq x_{1}\leq\cdots\leq x_{n\leq\cdots}.\end{equation} This means that
$\{x_{n}\}$ is an increasing sequence.
So
\begin{align*}\label{eq4}
\theta &\leq x_{n+1}-x_n \\
   &= Bx_n-Bx_{n-1} \\
   &\leq f(Nf(\|x_n-x_{n-1}\|)\|x_n-x_{n-1}\|)f(\|x_n-x_{n-1}\|)(x_n-x_{n-1}).
\end{align*}
Since $P$ is normal, from Lemma 1.5 we have
\begin{equation*}
  \|x_{n+1}-x_n\|_1\leq f(Nf(M\|x_n-x_{n-1}\|_1)M\|x_n-x_{n-1}\|_1)f(M\|x_n-x_{n-1}\|_1)\|x_n-x_{n-1}\|_1.
\end{equation*}
Since $f(t)\in (0,1)$ for all $t\geq0$, so $\|x_{n+1}-x_n\|_1 \leq\|x_n-x_{n-1}\|_1$, i.e., $\{\|x_n-x_{n-1}\|_1\}(n=1,2,\cdots)$
is a nonnegative decreasing sequence. From $f$ is increasing we know
\begin{equation*}
  f(M\|x_n-x_{n-1}\|_1)\leq f(M\|x_1-x_0\|_1)<1.
\end{equation*}
So
\begin{equation*}
  \|x_{n+1}-x_n\|_1\leq f(Nf(M\|x_1-x_0\|_1)M\|x_1-x_0\|_1)f(M\|x_1-x_0\|_1)\|x_n-x_{n-1}\|_1.
\end{equation*}
Let $\lambda=f(Nf(M\|x_1-x_0\|_1)M\|x_1-x_0\|_1)f(M\|x_1-x_0\|_1),$ then $\lambda\in(0,1).$
So
\begin{equation*}
  \|x_{n+1}-x_n\|_1\leq \lambda\|x_n-x_{n-1}\|_1\leq...\leq \lambda^n\|x_1-x_0\|_1.
\end{equation*}
We can assert that $\{x_{n}\}$ is a Cauchy sequence in $(E, \parallel.\parallel).$ In fact, for any positive integar $n,m$,
\begin{align*}
  \parallel x_{n+m}-x_{n}\parallel_{1} &\leq \parallel x_{n+m}-x_{n+m-1}\parallel_{1}+\cdots+ \parallel x_{n+1}-x_{n}\parallel_{1} \\
   &\leq (\lambda^{n+m-1}+\cdots+\lambda^{n})\parallel x_{1}-x_{0}\parallel_{1} \\
   &\leq \frac{\lambda^{n}}{1-\lambda}\parallel x_{1}-x_{0}\parallel_{1}.
\end{align*}
It follows in a standard way that $\{x_{n}\}$ is a Cauchy sequence
in $E.$ Since $E$ is complete, we can suppose that $x_{n}\rightarrow x_{\ast}\in E.$ (\ref{eq5}) together with Lemma 1.5 implies that
 \begin{equation}\label{eq6}
 x_{n}\leq x_{\ast}.
 \end{equation}
 (\ref{eq6}), together with (\ref{eq7}) and the equivalence of $\parallel.\parallel_{1}$ and $\parallel.\parallel$ implies that
\begin{equation*}
\parallel Bx_{\ast}-Bx_{n}\parallel_{1}\leq f(Nf(M\|x_{\ast}-x_{n}\|_1)M\|x_{\ast}-x_{n}\|_1)f(M\|x_{\ast}-x_{n}\|_1)\|x_{\ast}-x_{n}\|_1.
\end{equation*}
So
\begin{align*}
  \parallel x_{\ast}-Bx_{\ast}\parallel_{1} &\leq \parallel x_{\ast}-x_{n+1}\parallel_{1}+\parallel Bx_{\ast}-x_{n+1}\parallel_{1}\\
 &=\parallel x_{\ast}-x_{n+1}\parallel_{1}+\parallel Bx_{\ast}-Bx_{n}\parallel_{1} \\
   &\leq \parallel x_{\ast}-x_{n+1}\parallel_{1}\\
   &+f(Nf(M\|x_{\ast}-x_{n}\|_1)M\|x_{\ast}-x_{n}\|_1)f(M\|x_{\ast}-x_{n}\|_1)\|x_{\ast}-x_{n}\|_1.
\end{align*}
Let $n\rightarrow\infty,$ we obtain $\parallel x_{\ast}-Bx_{\ast}\parallel_{1}=0.$ So $x_{\ast}=Bx_{\ast}$, i.e., $x_{\ast}$ is a fixed point of $B$ in $E$
and $\lim_{n\rightarrow \infty}Bx_{n}=\lim_{n\rightarrow \infty}B^{n}x_{0}=Bx_{\ast}=x_{\ast}$.

Then we will prove the uniqueness of the fixed point. On the contrary, if $\overline{x}$ is another fixed point of $B,$ we will get $\overline{x}=x_{\ast}.$
In fact, the first case, when $\overline{x}$ is comparable with $x_0.$ Without loss of generality, we suppose that $\overline{x}\leq x_0.$
Since $B$ is increasing, $B^{n}\overline{x}\leq B^{n}x_{0}$. Similar to the proof of the monotonicity of the sequence $\{\|x_n-x_{n-1}\|_1\}(n=1,2,3,\cdots)$, we can obtain
$\{\parallel B^{n}\overline{x}-B^{n}x_{0}\parallel_{1}\}(n=0,1,2,\cdots)$ is also a increasing sequence and
\begin{equation*}
\parallel B^{n}\overline{x}-B^{n}x_{0}\parallel_{1}\leq \lambda_{1}\parallel B^{n-1}\overline{x}-B^{n-1}x_{0}\parallel_{1}\leq\cdots\leq\lambda_{1}^{n}\parallel\overline{x}-x_{0}\parallel_{1},
\end{equation*}
in which $\lambda_1=f(Nf(M\|\overline{x}-x_0\|_1)M\|\overline{x}-x_0\|_1)f(M\|\overline{x}-x_0\|_1),$ then $\lambda_{1}\in(0,1).$ Let $n\rightarrow\infty,$ we have
\begin{equation}\label{eq11}
\overline{x}=\lim_{n\rightarrow \infty}B^{n}\overline{x}=\lim_{n\rightarrow \infty}B^{n}x_{0}=x_{\ast}.
\end{equation}The second case, when $\overline{x}$ can not compare with $x_0.$ From (\ref{eq1}), we obtain $x_1=\inf\{\overline{x},x_0\}, x_2=\sup\{\overline{x},x_0\} \in E$ satisfying
\begin{equation*}
x_{1}\leq \overline{x}\leq x_{2}, x_{1}\leq x_{0}\leq x_{2}
\end{equation*}
i.e., $\overline{x}$ is comparable with $x_1, x_2$ and $x_0$ is comparable with $x_1, x_2$.
Since $B$ is increasing, we know
\begin{equation*}
Bx_{1}\leq B\overline{x}\leq Bx_{2}, Bx_{1}\leq Bx_{0}\leq Bx_{2},
\end{equation*}
and for any natural number $n$
\begin{equation*}
B^{n}x_{1}\leq B^{n}\overline{x}\leq B^{n}x_{2}, B^{n}x_{1}\leq B^{n}x_{0}\leq B^{n}x_{2}.
\end{equation*}
 So we know $B^{n}\overline{x}$ can compare with $B^{n}x_{1}$ and
$B^{n}x_{2}$. Similarly we can prove that $\{\parallel B^nx_i-B^nx_0\parallel_{1}\}(n=0,1,2,\cdots)$ and $\{\|B^nx_i-B^n\overline{x}\|_1\}(n=0,1,2,\cdots)$ are also
nonnegative and decreasing sequences(in which $B^0x=x$) and
\begin{equation}\label{eq8}
\|B^nx_i-B^nx_0\|_1\leq\lambda_{2i}\|B^{n-1}x_i-B^{n-1}x_0\|_1\leq...\leq\lambda^n_{2i}\|x_i-x_0\|,
\end{equation}
\begin{equation}\label{eq9}
\|B^nx_i-B^n\overline{x}\|_1\leq\lambda_{3i}\|B^{n-1}x_i-B^{n-1}\overline{x}\|_1\leq...\leq\lambda^n_{3i}\|x_i-\overline{x}\|, i=1,2,
\end{equation}
in which
\begin{equation*}
  \lambda_{2i}=f(Nf(M\|x_i-x_0\|_1)M\|x_i-x_0\|_1)f(M\|x_i-x_0\|_1),
\end{equation*}
\begin{equation*}
  \lambda_{3i}=f(Nf(M\|\overline{x}-x_i\|_1)M\|\overline{x}-x_i\|_1)f(M\|\overline{x}-x_i\|_1),
\end{equation*}
$\lambda_{2i}, \lambda_{3i}\in(0,1).$
Let $n\rightarrow\infty$ in (\ref{eq8}) and (\ref{eq9}), we have
\begin{equation*}
\lim\limits_{n\to\infty}B^nx_i=\lim\limits_{n\to\infty}B^nx_0=x_{\ast},
\end{equation*}
\begin{equation*}
\lim\limits_{n\to\infty}B^nx_i=\lim\limits_{n\to\infty}B^n\overline{x}=\overline{x}.
\end{equation*}
So \begin{equation}\label{eq10}
\overline{x}=x_{\ast}.
\end{equation}

 (\ref{eq11}) together with (\ref{eq10}) implies that $x_{\ast}$ is unique fixed point of $B.$

Next we will prove that the unique fixed point of $B$ is also the unique fixed point of $A.$

 Since
\begin{equation*}
A^2x_{\ast}=Bx_{\ast}=x_{\ast},
\end{equation*}and
\begin{equation*}
A^2(Ax_{\ast})=A(A^2x_{\ast})=Ax_{\ast},
\end{equation*}
i.e., $B(Ax_{\ast})=Ax_{\ast}.$ From the uniqueness of the fixed point of $B$ we know
\begin{equation}
Ax_{\ast}=x_{\ast}.
\end{equation}
 So $x_{\ast}$ is the unique fixed point of $A$ in $E.$

 \textbf{Case II}: Another case, when $u_{0}$ is not comparable to $Au_{0.}$
From the assumption (H), we know
 there exists $v_0\in E$ such that $\inf\{Au_0, u_0\}=v_0.$ That is $v_0\leq Au_0, v_0\leq u_0.$ Since A is a decreasing operator, we have
\begin{equation*}
A^2u_0\leq Av_0, Au_0\leq Av_0.
\end{equation*}
This shows that
\begin{equation}
v_0\leq Av_0.
\end{equation} Similarly as the proof of Case I, we can get that $A$ has unique fixed
point in E.
\end{proof}

\section{Applications}
In this section, we present two examples where our Theorem can be applied.
\begin{ex}
We consider the self-feedback stability of a signal outlet function in nonlinear suppressed interference channel. When outlet signals are fed back to
the input process, we want to know whether the final signals are stable.
We suppose that the signal period is $1$. We only consider situations in a period. The signal space is $C[0,1]$ and the signal output
function is(only in the case of real number)
\begin{equation*}
  Au(t)=\frac{1}{2\pi+u(t)}-\frac{\pi^{2}}{16}\int_{0}^{1}(s^2+t^2)\frac{1+u(t)s^2}{2\pi M}ds,
\end{equation*}
$M$ is a positive integer. Let $P=\{u(t)|u(t)\geq0,\ t\in[0,1]\}$, then $P$ is a normal cone in $C[0,1].$
The partial order $\leq$ induced by $P$ is: $u\leq v\Leftrightarrow u(t)\leq v(t)$ for all $t\in [0,1].$
$E=C[0,1]$ is a partially ordered Banach space. Evidently, $A$ is a decreasing operator. For all
$u(t),\ v(t)\in E$ satisfying $u(t)\leq v(t)$, we obtain that
\begin{align*}
 Au(t)-Av(t) &= \frac{1}{2\pi+u(t)}-\frac{\pi^{2}}{16}\int_{0}^{1}(s^2+t^2)\frac{1+u(t)s^2}{2\pi M}ds \\
   &- \frac{1}{2\pi+v(t)}+\frac{\pi^{2}}{16}\int_{0}^{1}(s^2+t^2)\frac{1+v(t)s^2}{2\pi M}ds\\
   &= \frac{v(t)-u(t)}{(2\pi+u(t))(2\pi+v(t))}+\frac{\pi^{2}}{16}\int_{0}^{1}(s^2+t^2)\frac{(v(t)-u(t))s^2}{2\pi M}ds \\
   &\leq \frac{v(t)-u(t)}{4\pi^2}+\frac{\pi}{60M}(v(t)-u(t))\\
   &\leq \frac{3}{20}(v(t)-u(t)).
\end{align*}
When take $f(t)=\frac{3}{20}$ in Theorem 2.1, it is easy to know that the conclusion of Theorem 2.1 holds, i.e., there is unique $u^{\ast}\in E$ such that $Au^{\ast}=u^{\ast}$. This means that the signal outlet function has
self-feedback stability.
\end{ex}
\begin{ex}
 Now, we study the existence of solution for the following first-order periodic problem
\begin{equation}
\begin{cases}
u'(t)=F(t,u(t)), t\in[0,1]\\
u(0)=u(1).
\end{cases}
\end{equation}
 where $F:[0,1]\times R\rightarrow R$ is a continuous function.

We consider the space $C(I)(I=[0,1])$ of continuous functions defined on $[0,1]$. Obviously, this space with the metric given by
\begin{equation*}
d(x,y)=\sup|x(t)-y(t)|,t\in I, x,y\in C(I)
\end{equation*}
ia a Banach space. $C(I)$ can be equipped with a partial order induced by a cone
\begin{equation*}
P=\{y-x:y(t)-x(t)\geq0, t\in I\}.
\end{equation*}
Obviously, P is a normal cone and assume that its normal constant is N. And the order relation in $C(I)$ induced by $P$ is:
\begin{equation*}
x,y\in C(I), x\leq y \Leftrightarrow x(t)\leq y(t), t\in I.
\end{equation*}

\begin{thm}
Consider problem (3.1) with $F:I\times R\rightarrow R$ continuous and suppose that there exists $\lambda>0$ and $0<\alpha\leq \lambda N$ such that for $x,y\in R$ with $x\geq y$,
\begin{equation*}
0\leq F(t,y)+\lambda y-(F(t,x)+\lambda x)\leq \alpha (x-y)\ln[(1+\frac{1}{x-y})^{x-y}].
\end{equation*}
Then there exists unique solution for problem (3.1).
\end{thm}
\begin{proof}
 Problem (3.1) can be written as
\begin{equation*}
u(t)=\int_{0}^{1}G(t,s)[F(s,u)+\lambda u(s)]ds,
\end{equation*}
where $G(t,s)$ is the Green function given by
$$G(t,s)=\begin{cases}
\frac{e^{\lambda(1+s-t)}}{e^{\lambda }-1}, &0\leqslant s<t\leqslant 1;\\
\frac{e^{\lambda(s-t)}}{e^{\lambda }-1}, &0\leqslant t<s\leqslant.
\end{cases}$$

Define $T:C(I)\rightarrow C(I)$by
\begin{equation*}
(Tu)(t)=\int_{0}^{1}G(t,s)[F(s,u)+\lambda u(s)]ds.
\end{equation*}
Note that if $u\in C(I)$ is a fixed point of $T$ then $u\in C^{1}(I)$ is a solution of (3.1).

In what follows, we check that hypotheses of Theorem 2.1 are satisfied.

Clearly, $(C(I), \leq)$ satisfies condition (\ref{eq}), since for $x,y\in C(I)$ the functions $\max\{x,y\}, \min\{x,y\}$ are least upper and greatest lower bounds of $x$ and $y$, respectively.

The operator $T$ is decreasing, since for $u\geq v,$ and using our assumption, we can obtain
\begin{equation*}
F(t,u)+\lambda u\leq F(t,v)+\lambda v,
\end{equation*}
which implies, since $G(t,s)>0,$ that for $t\in I,$
\begin{align*}
(Tu)(t)&=\int_{0}^{1}G(t,s)[F(s,u(s))+\lambda u(s)]ds\\
&\leq \int_{0}^{1}G(t,s)[F(s,v(s))+\lambda v(s)]ds=(Tv)(t),
\end{align*}
Besides, for $u\geq v,$ we have
\begin{align*}
 \|Tu-Tv\| &=\sup_{t\in I}\mid (Tu)(t)-(Tv)(t)\mid\\
 &\leq \sup \int_{0}^{1}G(t,s)\mid F(s,u(s))+\lambda u(s)-F(s,v(s))-\lambda v(s)\mid ds\\
   &\leq \sup \int_{0}^{1}\alpha G(t,s)(u-v)\ln[(1+\frac{1}{u-v})^{u-v}]ds \\
   &\leq \alpha \|u-v\| f(\|u-v\|)\sup \int_{0}^{1}G(t,s)ds\\
   &=\alpha \|u-v\| f(\|u-v\|)\sup_{t\in I}\frac{1}{e^{\lambda}-1}[\frac{1}{\lambda}e^{\lambda(1+s-t)}|_{0}^{t}+\frac{1}{\lambda}e^{\lambda(s-t)}|_{t}^{1}]\\
   &=\alpha \|u-v\| f(\|u-v\|)\frac{1}{\lambda(e^{\lambda}-1)}(e^{\lambda}-1)\\
   &=\alpha\|u-v\| f(\|u-v\|)\frac{1}{\lambda}\\
   &\leq Nf(\|u-v\|)\|u-v\|.
\end{align*}
This implies that $T$ satisfies condition (\ref{eq*}) which can be used to prove the uniqueness of solution. And (\ref{eq*}) is deduced by (\ref{eq1}).

In the above inequalities we choose $f(t)=t\ln(1+\frac{1}{t})$. It is easy to
 prove that $f(t)$ is increasing and $f(t):(0,+\infty)\rightarrow(0,1)$.

Finally, Theorem 2.1 gives that $T$ has an unique fixed point.
\end{proof}
\end{ex}

\subsection*{Competing Interests} The authors declare that they have no competing interests.

\subsection*{Authors' Contributions} J. Mao proved main Theorems. Z. Zhao gave the application of
the manuscript. All authors read and approved the manuscript.

\subsection*{Acknowledgment} The authors would like to thank the referee for his/her careful reading and valuable suggestions.

\end{document}